\documentclass[11pt]{article}
\usepackage[utf8]{inputenc}
\usepackage{bm}
\usepackage{extpfeil}
\usepackage{setspace}
\usepackage{tikz}
\usepackage{siunitx}
\usepackage{blindtext, rotating}
\usepackage{booktabs}
\usepackage{array}
\newcolumntype{C}[1]{>{\centering\arraybackslash}p{#1}}
\usepackage{color}
\usepackage{float}
\usetikzlibrary{matrix,positioning}
\usepackage{mathtools}

\makeatother

\usepackage{amsmath,tabu}
\usepackage{tikz-cd}
\usepackage[utf8]{inputenc}
\usepackage{ctable}
\usepackage{pdflscape}
\usepackage{array}
\newcolumntype{P}[1]{>{\centering\arraybackslash}p{#1}}
\usepackage{enumerate}
\usepackage{amsthm}
\usepackage{amsfonts}
\usepackage{multicol}
\usepackage{array}
\newcolumntype{L}{>{$}l<{$}}
\usepackage{hyperref}
\usepackage[margin=1in]{geometry}
\usepackage{enumitem} 
\newtheorem{theorem}{Theorem}[section]
\newtheorem{thm}{Theorem}

\newtheorem{cor}[thm]{Corollary}
\newtheorem{lemma}[theorem]{Lemma}
\newtheorem{proposition}[theorem]{Proposition}
\newtheorem{conjecture}{Conjecture}
\usepackage{amssymb}
\newtheorem{hypothesis}[theorem]{Hypothesis}

\theoremstyle{definition}

\newtheorem{notation}[theorem]{Notation}

\newtheorem{definition}[theorem]{Definition}
\usepackage{graphicx}
\usepackage{hyperref}
\newcommand{\pb}{{p_k}}
\newcommand{\pk}{{p^{(1)}}}
\newcommand{\pnk}{{p^{(2)}}}

\newcommand{\qr}{{q}}
\newcommand{\A}[1]{\mathrm{A}_{#1}} 
\renewcommand{\S}[1]{\mathrm{S}_{#1}} 
\newcommand{\Sym}[1]{\mathrm{Sym}({#1})} 
 
\newcommand{\PSL}{\mathrm{PSL}} 
\newcommand{\Syl}{\mathrm{Syl}} 
\newcommand{\AGL}{\mathrm{AGL}} 
\newcommand{\PGamL}{\mathrm{P\Gamma L}} 
\newcommand{\supp}[1]{\mathrm{Supp}({#1})} 
\newcommand{\fix}[1]{\mathrm{Fix}({#1})} 
\newcommand{\stab}[2]{\mathrm{Stab}_{#1}({#2})} 
\newcommand{\x}[1]{x} 
\newcommand{\y}[1]{y} 
\newcommand{\ctm}[1]{\mathcal{C}_M({#1})}
\newcommand{\ct}[1]{\mathcal{C}({#1})}
\usepackage{bigfoot}
\begin{document}
\title{{\Large}{\textbf{Maximal Cocliques in the Generating Graphs of the Alternating and Symmetric Groups}}}

\author{Veronica Kelsey and Colva M. Roney-Dougal \thanks{ 
\textit{Acknowledgements:} The authors would like to thank the Isaac Newton Institute for Mathematical Sciences for support and hospitality during the programme \textit{Groups, representations and applications: new perspectives}, when work on this paper was undertaken. This work was supported by: EPSRC grant number EP/R014604/1. In addition, this work was partially supported by a grant from the Simons Foundation.  \newline \indent
}}
\maketitle
\begin{abstract} The \textit{generating graph} $\Gamma(G)$ of a finite group $G$ has vertex set the non-identity elements of $G$, with two elements connected exactly when they generate $G$. A \textit{coclique} in a graph is an empty induced subgraph, so a coclique in $\Gamma(G)$ is a subset of $G$ such that no pair of elements generate $G$. A coclique is \textit{maximal} if it is contained in no larger coclique. It is easy to see that the non-identity elements of a maximal subgroup of $G$ form a coclique in $\Gamma(G)$, but this coclique need not be maximal.

In this paper we determine when the intransitive maximal subgroups of $\S{n}$ and $\A{n}$ are maximal cocliques in the generating graph.
In addition, we prove a conjecture of Cameron, Lucchini, and Roney-Dougal \cite{Colva} in the case of $G = \A{n}$ and $\S{n}$, when $n$ is prime and $n \neq \frac{q^d-1}{q-1}$ for all prime powers $q$ and $d \geq 2$. Namely, we show that two elements of $G$ have identical sets of neighbours in $\Gamma(G)$ if and only if they belong to exactly the same maximal subgroups.
\end{abstract}
\section{Introduction}
The \textit{generating graph} $\Gamma(G)$ of a group $G$ has vertex set the non-identity elements of $G$, with two elements connected exactly when they generate $G$. A subset of vertices in a graph forms a \emph{coclique} if no two vertices in the subset are adjacent. A coclique is \emph{maximal} if it is contained in no larger coclique. 

The definition of a generating graph was first introduced by Liebeck and Shalev in \cite{Shalev}. Let $m(G)$ denote the minimum index of a proper subgroup of $G$. Liebeck and Shalev showed that for all $c<1$, if $G$ is a sufficiently large simple group, then $\Gamma(G)$ contains a clique of size at least $cm(G)$. That is, $G$ contains a subset $S$ of size at least $cm(G)$ such that all two-element subsets of $S$ generate $G$. 
See \cite{ref3}, \cite{ref1} and \cite{ref2} for more results about cliques in generating graphs.

Less is known about cocliques in generating graphs. In a slight abuse of language, we shall refer to maximal subgroups as cocliques in $\Gamma(G)$, even though strictly speaking it is their non-identity elements that form a coclique. Recently in \cite{Jack}, Saunders proves that for each odd prime $p$, a maximal coclique in $\Gamma(\PSL_2(p))$ is either a maximal subgroup, the conjugacy class of all involutions, or has size at most $\frac{129}{2}(p-1)+2$.

This paper determines when an intransitive maximal subgroup $M$ of $G=\S{n}$ or $G=\A{n}$ is a maximal coclique in $\Gamma(G)$. In each case we either show that $M$ is a maximal coclique or describe the maximal coclique containing $M$.

Our first main result is the following.

\begin{thm} \label{newmain}
Let $n \geq 4$, let $G = \S{n}$ or $\A{n}$, let $n > k > \frac{n}{2}$ and let $M =(\S{k} \times \S{n-k}) \cap G$ be an intransitive maximal subgroup of $G$. 
\begin{enumerate}[label=\rm{(\roman*)}]
\item If $G = \S{n}$, then $M$ is a maximal coclique in $\Gamma(G)$ if and only if $\gcd(n,k)=1$ and $(n,k) \neq (4,3)$.
\item If $G = \A{n}$, then $M$ is a maximal coclique in $\Gamma(G)$ if and only if $(n,k) \notin \{(5,3), (6,4)\}$.
\end{enumerate}
\end{thm}
Our second main theorem concerns the exceptional cases of Theorem \ref{newmain}.
\begin{thm} \label{mainsize}
\begin{enumerate}[label=\rm{(\roman*)}]
\item Let $n \geq 4$, let $G = \S{n}$, let $n > k > \frac{n}{2}$ and let $M =\S{k} \times \S{n-k}$ be an intransitive maximal subgroup of $G$, setwise stabilising $\{1, \ldots, k\}$.
\begin{enumerate}[label=\rm{(\Roman*)}]
\item[\rm{(a)}] If $\gcd(n,k)>1$, then the unique maximal coclique of $\Gamma(G)$ containing $M$ is $M \cup (1, k+1)^M \backslash \{1\}.$
\item[\rm{(b)}] If $(n,k)=(4,3)$, then the unique maximal coclique of $\Gamma(G)$ containing $M$ is $M \cup (1,4)(2,3)^M \backslash \{1\}$.
\end{enumerate}
\item Let $(n,k) \in \{(5,3), (6,4)\}$, let $G = \A{n}$ and let $M =(\S{k} \times \S{n-k}) \cap G$ be an intransitive maximal subgroup of $G$. 
\begin{enumerate}[label=\rm{(\Roman*)}]
\item[\rm{(a)}] If $(n,k) = (5,3)$, then the unique maximal coclique of $\Gamma(G)$ containing $M$ is 
$M \cup (1,4)(2,3)^M \backslash \{1\}.$
\item[\rm{(b)}] If $(n,k) = (6,4)$, then the unique maximal coclique of $\Gamma(G)$ containing $M$ is 
$M \cup (1,5)(2,6)^M \backslash \{1\}.$
\end{enumerate}
\end{enumerate}
\end{thm}

In \cite{Colva}, Cameron, Lucchini and Roney-Dougal define an equivalence relation $\equiv_m$ and a chain of equivalence relations $\equiv_m^{(r)}$ on the elements of a finite group $G$. Two elements $x,y \in G$ satisfy $x \equiv_m y$ exactly when $x$ and $y$ can be substituted for one another in all generating sets for $G$. Equivalently, $x \equiv_m y$ when $x$ and $y$ lie in exactly the same maximal subgroups of $G$. Conversely, $x \equiv_m^{(r)} y$ when $x$ and $y$ can be substituted for one another in all generating sets for $G$ of size $r$. The relations $\equiv_m^{(r)}$ become finer as $r$ increases, with limit $\equiv_m$, and $\psi(G)$ is defined to be the smallest value of $r$ for which $\equiv_m$ and $\equiv_m^{(r)}$ coincide. 

\begin{conjecture}[{\cite[Conjecture 4.7]{Colva}}\label{Conj}] Let $G$ be a finite group such that no vertex of $\Gamma(G)$ is isolated. Then $\psi(G) \leq 2$.
\end{conjecture}
Settling a long-standing conjecture, Burness, Guralnick and Harper show in \cite{Scott} that if $G$ is a finite group and all quotients of $G$ are cyclic, then no vertex of $\Gamma(G)$ is isolated. The result for $G = \A{n}$ and $\S{n}$ goes back much further, see \cite{CaptainPicard}.

Cameron, Lucchini and Roney-Dougal observe in \cite{Colva} that to prove this conjecture, it suffices to show that each maximal subgroup is a maximal coclique in $\Gamma(G)$. This motivates the following theorem.
\begin{thm} \label{AGLstuff} Let $p \geq 5$ be a prime such that $p \neq  \frac{q^d-1}{q-1}$ for all prime powers $q$ and all $d \geq 2$. Let $G = \S{p}$ or $\A{p}$. 
\begin{enumerate}[label=\rm{(\roman*)}]
\item If $G = \S{p}$, then each maximal subgroup of $G$ is a maximal coclique in $\Gamma(G)$.
\item If $G=\A{p}$, then each maximal subgroup $M$ of $G$ is a maximal coclique in $\Gamma(G)$ except when $p=5$ and $M = (\S{3} \times \S{2}) \cap G$.
\end{enumerate}
\end{thm}
Theorem 2.26 of \cite{Colva} states that $\psi(\A{5})=2$. Hence the following is immediate.
\begin{cor} \label{cor} Let $G$ and $p$ be as in Theorem \ref{AGLstuff}. Then $\psi(G) = 2$. That is, two elements of $G$ belong to exactly the same maximal subgroups of $G$ if and only if they can be substituted for each other in all generating pairs for $G$.
\end{cor}
This paper is structured as follows. In Section 2 we begin with some background results on number theory, cycle structures of elements of $\S{n}$ and block systems of imprimitive permutation groups.
In Section 3 we show that Theorems \ref{newmain} and \ref{mainsize} hold for $n \leq 11$ and prove some preliminary lemmas. In Section 4 we complete the proof of Theorems \ref{newmain} and \ref{mainsize}. Finally, in Section 5 we prove Theorem \ref{AGLstuff}.
\section{Background Results}
\subsection{Number Theoretical Background}
In this subsection we collect results about the existence of primes in certain subsets of the integers. We start with Bertrand's Postulate. Throughout this subsection, all $\log$s are natural logarithms.
\begin{theorem}[Bertrand's Postulate. See for example {\cite[§1]{BP}}\label{BPC}] Let $m \in \mathbb{N}$. If $m \geq 4$, then there exists at least one prime $p$ such that $m < p < 2m-2$. Hence for $k \in \mathbb{N}$ with $k \geq 7$, there exists a prime $\pb \geq 5$ with $\frac{k}{2} < \pb < k - 1$.
\end{theorem}
\begin{notation} \label{notation} For $k \in \mathbb{N}$ with $k \geq 7$, let $\pb$ denote a prime as in Theorem \ref{BPC}.
\end{notation}
We note that $\pb$ does not divide $k$, and that $\pb$ is not uniquely determined by $k$, but at least one such prime must exist.

The proof of the following lemma is straightforward. 
\begin{lemma} \label{coprime} Let $n > k > \frac{n}{2}$ with $k \geq 7$ and let $\pb$ be as in Notation \ref{notation}. If $\pb \mid (n-k)$ then $\pb = n-k$, and if $\pb \mid (n-k-1)$ then $\pb = n-k-1$.
\end{lemma}
We will need two variations of Bertrand's Postulate.
\begin{lemma} \label{pk} Let $n > k > \frac{n}{2}$, with $k \geq 10$. Then there exists an odd prime $\pk \leq k-5$ such that $\pk \nmid (n-k)$.
\end{lemma}
\begin{proof}
Let $Q = \{q\text{ prime} : 2 \leq q \leq k-5\}$. The product of the set of prime divisors of $n-k$ is at most $n-k$, so if 
\begin{equation} \label{showthat}
2(n-k) < \prod_{q \in Q} \qr,
\end{equation}
then there exists an odd prime $\pb \in Q$, as required.

Since $k \geq 10$, the set $Q$ contains $\{2,3,5\}$ and so $\prod_{q \in Q} \qr \geq 30$. If $ k \leq 15$, then $n-k \leq k-1 \leq 14$. Hence \eqref{showthat} holds for $10 \leq k \leq 15$.

Assume from now on that $k >15$, and set $m=k-5>10$. Applying Theorem \ref{BPC} with $m$ in place of $k$ provides a prime $p_m$ with $5 < \frac{m}{2} < p_m <m-1$. Hence $2,3,5$ and $p_m$ are in $Q$. Observe also that $15m > 2(m+4)$ and $m+4=k-1 \geq n-k$. Hence 

$$2(n-k) \leq 2(m+4)<15m < 3 \cdot 5 \cdot (2p_m)  \leq \prod_{q \in Q} \qr,$$
as required.
\end{proof}
\begin{lemma} \label{pnk} Let $n > k > \frac{n}{2}$. If $n-k>10$, then at least one of the following holds.
\begin{enumerate}[label=\rm{(\roman*)}]
\item There exists a prime $\pnk$ with $2< \pnk < n-k-3$, such that $\pnk \nmid k$.
\item The inequality $n-k+1 < 2(\sqrt{n}-1)$ holds.
\end{enumerate}
\end{lemma}
\begin{proof} First suppose that $10 < n-k < 26$ and let $P = \{q \text{ prime} : 2<q<n-k-3\}$. If (i) does not hold, then all primes in $q \in P$ divide $k$, and hence $ \prod_{q \in P} q  \leq k < n$. For $10<n-k<26$ a straightforward calculation shows that
$$\frac{(n-k+3)^2}{4}< \prod_{q \in P} q ,$$
and so $(n-k+3)^2/4< n.$ 
Rearranging gives the desired inequality in (ii).\\

Now suppose that $n-k \geq 26$. Let $m=n-k-3$, so that $m \geq 23$, and let $\pi(m)$ be the number of primes less than or equal to $m$. We shall first prove that
\begin{equation} \label{whatever2}
2\Big(\pi(m-1)-4\Big) >  \log \Bigg(2\Big(\frac{m}{2}+3\Big)^2\Bigg). 
\end{equation}
To do so let $y:=y(m)$ be the following function of $m$
$$y = (m-1)  -\log \Big(\frac{m}{2} +3 \Big) \log(m-1) - \frac{1}{2}\Big( \log(2) +8 \Big)\log(m-1).$$
Then
$$\frac{dy}{dm} = 1 - \frac{\log \Big(\frac{m}{2} +3 \Big)}{m-1}-\frac{\log(m-1) }{m +6 }-\frac{\log(2)+8}{2(m-1)}.$$
The functions $\frac{\log(\frac{m}{2}+3)}{m-1}$ and $\frac{\log(2)+8}{2(m-1)}$ are monotonically decreasing for $ m \geq 2$, the function $\frac{\log(m-1)}{m+6}$ is monotonically decreasing for $m \geq 9$, and $\frac{dy}{dm}$ is positive at $m=9$. Hence $\frac{dy}{dm}$ is positive for $m \geq 9$. Since $y(23)>0$, it follows that $y$ is positive for $m \geq 23$. Hence for $m \geq 23$
$$(m-1) -4\log(m-1) > \log(\frac{m}{2}+3) \log(m-1) + \frac{1}{2}\log(2) \log(m-1),$$
and so
\begin{equation} \label{whatever}
2\Big( \frac{m-1}{\log(m-1)} -4 \Big)>  2 \log \Big( \frac{m}{2} +3 \Big) + \log(2) =  \log \Bigg(2\Big(\frac{m}{2}+3\Big)^2\Bigg).
 \end{equation}
Corollary 1 of \cite{rosser} states that $\pi(x) > \frac{x}{\log(x)}$ for $x \geq 17$. 
Hence \eqref{whatever} implies \eqref{whatever2}.\\

\medskip
Let $Q = \{q \text{ prime}: 2 \leq q < m\}$ and $Q_0 = \{q \in Q : q> 7\}.$ Observe that if $q \in Q_0$, then $\log(q) > 2$. Therefore
$$\log \Bigg( \prod_{q \in Q} q  \Bigg)= \sum\limits_{q \in Q} \log(q) >  \sum\limits_{q \in Q_0} 2 =
2\Big(\pi(m-1)-4\Big)>  \log \Bigg(2\Big(\frac{m}{2}+3\Big)^2\Bigg).$$
Thus 
$$\prod_{q \in Q} q > 2\Big(\frac{m}{2}+3\Big)^2.$$
If (i) does not hold, then $q \mid k$ for all odd primes $q \in Q$. Then $k$ is greater than or equal to the product of all such primes, so 
$$2n> 2k \geq \prod_{q \in Q} q >  2\Big(\frac{m}{2}+3\Big)^2.$$
Hence $\sqrt{n} > \frac{m}{2}+3$ and so
$$2(\sqrt{n}-1) > m+4 = (n-k-3)+4 = n-k+1,$$
as in (ii). Hence the lemma holds. 
\end{proof}
\subsection{Elementary Results on Cycle Structures and Primitivity}
This subsection collects several technical results concerning cycle structures, primitive groups and block systems.
\begin{definition} \label{Jor} For $n \geq 12$ we refer to the following elements of $\S{n}$ as \textit{Jordan elements}:
\begin{enumerate}[label=\rm{(\roman*)}]
\item products of two transpositions; \label{Jordan2}
\item cycles fixing at least three points; \label{Jones} 
\item permutations with support size less than or equal to $2(\sqrt{n}-1)$. \label{LS}
\end{enumerate}
\end{definition}
The following result will be used extensively in the rest of the paper.
\begin{theorem} \label{jordanstyle}
Let $G \leq \S{n}$ be primitive. If $G$ contains a Jordan element, then $\A{n} \leq G$.
\end{theorem}
\begin{proof}  
Types (i), (ii) and (iii) from Definition \ref{Jor} are dealt with by page 43 of \cite{Wielandt}, Corollary 1.3 of \cite{Jones} and Corollary 3 of \cite{LS} respectively.
\end{proof}
\begin{lemma} \label{oarity table} Let $y \in \S{n}$ be composed of $t$ (possibly trivial) disjoint cycles. Then $y$ is even if and only if $t$ and $n$ have the same parity.
\end{lemma}
\begin{proof} Let $y$ have $t_1$ cycles of odd length and $t_2$ cycles of even length, so that $t_1+t_2=t$. Then $n \equiv t_1\bmod 2$ so
$$t-n \equiv t-t_1 = t_2 \bmod 2.$$
Hence $t$ and $n$ have the same parity if and only if $t_2$ is even, that is if and only if $y$ is even.
\end{proof}
\begin{notation} Let $y \in \S{n}$ be composed of $t$ (possibly trivial) disjoint cycles $c_1 c_2 \ldots c_t$. For $1 \leq i \leq t$ let $\Theta_i = \supp{c_i}$. We denote the cycle type of $y$ by $\ct{y} = |c_1| \cdot |c_2| \cdot \ldots \cdot |c_t|$. Often the ``$\cdot$'' notation is omitted when it is clear without, and we sometimes gather together common cycle orders and use the usual exponent notation.
\end{notation}
For example, if $y = (1, 2, 3)(4, 5)(6, 7)$ then $c_1 = (1,2,3)$, $c_2 = (4,5)$ and $c_3 = (6,7)$. Thus $\Theta_1 = \{1,2,3\}$, $\Theta_2 = \{4,5\}$ and $\Theta_3 = \{6,7\}$, and we may choose to write $\mathcal{C}(y) = 3 \cdot 2 \cdot 2$ or $\mathcal{C}(y) = 3 \cdot 2^2$. \\

The next lemma guarantees under certain circumstances the existence of suitable sets of distinct points.
\begin{lemma} \label{8pts}
Let $\frac{n}{2}<k<n$, and let $x \in \S{n}$ be such that $1^x = k+1$. 
\begin{enumerate}[label=\rm{(\roman*)}]
\item If $|\supp{x}| \geq 8$ and $x$ does not have cycle type $1^{(n-8)} \cdot 2 \cdot 3^2,$ $1^{(n-8)} \cdot 3 \cdot 5$ or $1^{(n-9)} \cdot 3^3,$ then there exist distinct points $r, r^x, s, s^x ,t, t^x \in \supp{x} \backslash \{1, k+1\}$. \label{8pts1}
\item If $|\supp{x}| \geq 8$ and $x$ does not have cycle type $1^{(n-8)} \cdot 2^4,$ then there exist distinct points $s, s^x, t, t^x,u,v \in \supp{x} \backslash \{1, k+1\}$ such that $(u,v)$ is not a cycle of $x$.\label{8pts2}
\end{enumerate}
\end{lemma}
\noindent
\textit{Proof.}
Let $S = \supp{x}$ and $T = S \backslash 1^{\langle x \rangle}$. We split into cases based on $|1^{\langle x \rangle}|$. 
\begin{enumerate}[label=\rm{(\roman*)}]
\item If $|1^{\langle x \rangle}| \geq 8,$ then we may let $r = 1^{x^2}, s=1^{x^4}$ and $t=1^{x^6}$. If $6 \leq |1^{\langle x \rangle}| \leq 7,$ then $|T| \geq 2$. Let $r =1^{x^2}$, $s = 1^{x^4}$ and let $t \in T$. If $4 \leq |1^{\langle x \rangle}| \leq 5,$ then $|T| \geq 4$ because $x$ does not have cycle type $1^{(n-8)} \cdot 3 \cdot 5$. Hence either $\langle x \rangle$ has at least two orbits on $T$ of size at least 2 or one of size at least 4. Hence we may let $r=1^{x^2}$ and $s, t \in T$. If $ |1^{\langle x \rangle}| \leq 3,$ then $|T| \geq 6$ because $x$ does not have cycle type $1^{(n-8)} \cdot 3 \cdot 5$ or $1^{(n-8)} \cdot 2 \cdot 3^2$. Hence either $\langle x \rangle$ has one orbit on $T$ of size at least 6, or exactly two orbits, with sizes at least 3 and 4 respectively (because $x$ does not have cycle type $1^{(n-9)} \cdot 3^3$), or at least 3 orbits. Hence we may let $r, s, t \in T$. 
\item If $|1^{\langle x \rangle}| \geq 8,$ then let $u=1^{x^2}$, $v=1^{x^3}$, $s=1^{x^4}$ and $t = 1^{ x^6 }$.
If $6 \leq |1^{\langle x \rangle}| \leq 7,$ then let $u=1^{x^2}$, $v=1^{x^3}$, $s=1^{x^4}$ and let $t \in T$. The arguments for $|1^{\langle x \rangle} |\in \{3,4,5\}$ are straightforward. If $|1^{\langle x \rangle}|=2,$ then $|T| \geq 6$ and $\langle x \rangle$ does not have 3 orbits of size 2 on $T,$ since the cycle type of $x$ is not $1^{(n-8)} \cdot 2^4$. Hence we may let $u,v,s,t \in T$. \hfill \qed
\end{enumerate}

For the rest of this section, let $\Omega$ be a finite set and let $H$ be a transitive subgroup of $\Sym{\Omega}$ with a block system $\mathcal{B}$. We include the possibility of $\mathcal{B}$ being trivial, that is blocks of size 1 or $|\Omega|$.
\begin{notation} For $h_i$ a cycle of $h \in H$, let $h_i^{\mathcal{B}}$ be the permutation that $h$ induces on the set of blocks in $\mathcal{B}$ which contain elements of $\supp{h_i}$.
\end{notation}
\begin{lemma} \label{div} Let $h \in H$ with cycle $h_i$. Then $h_i^{\mathcal{B}}$ is a cycle whose length divides the length of $h_i$.
\end{lemma}
\begin{proof}
Since $h_i$ is transitive on the points of $\supp{h_i}$, it follows that $h_i^{\mathcal{B}}$ is a cycle. Let $\Delta$ be a block containing $m>0$ points of $\supp{h_i}$. It follows that each block of $\mathcal{B}$ contains exactly $m$ or 0 points of $\supp{h_i}$. Hence $|h_i| = m|h_i^{\mathcal{B}}|$.
\end{proof} 

\begin{lemma} \label{comb} Let $h \in H$ with disjoint (possibly trivial) cycles $h_1$ and $h_2$. 
\begin{enumerate}[label=\rm{(\roman*)}]
\item \label{sameinduced} Suppose that $\Delta$ is a block of $\mathcal{B}$ containing $\alpha \in \supp{h_1}$ and $\beta \in \supp{h_2}$. Then $h_1^{\mathcal{B}} = h_2^{\mathcal{B}}$.
\item \label{pcy} If $h_1$ has prime length $p$, then the points of $\supp{h_1}$ either lie in one block or each lie in different blocks.
\item \label{h1 cup h2} Suppose $h_1$ and $h_2$ have coprime lengths. If there exists a block $\Delta$ of $\mathcal{B}$ with $\supp{h_1} \cap \Delta \neq \emptyset$ and $\supp{h_2} \cap \Delta \neq \emptyset$, then $\supp{h_1} \cup \supp{h_2} \subseteq \Delta$.
\end{enumerate}
\end{lemma}
\noindent
\textit{Proof.}
\begin{enumerate}[label=\rm{(\roman*)}]
\item Since $\alpha, \beta \in \Delta$, it follows that for all $i$, the points $\alpha^{h^i}$ and $\beta^{h^i}$ lie in the same block. From $\alpha^{h^i} = \alpha^{h_1^{i}}$ and $\beta^{h^i} = \beta^{h_2^i}$, it follows that $h_1^{\mathcal{B}} = h_2^{\mathcal{B}}$.
\item By Lemma \ref{div}, $h_1^{\mathcal{B}}$ is either a $p$-cycle or a $1$-cycle.
\item By Part \ref{sameinduced}, $h_1^{\mathcal{B}} = h_2^{\mathcal{B}}$. Since $h_1$ and $h_2$ have coprime lengths, it follows from Lemma \ref{div} that $h_1^{\mathcal{B}}$ is trivial. \hfill \qed
\end{enumerate}
\begin{definition} Let $H$ be transitive, with block system $\mathcal{B}$, and let $\Delta \in \mathcal{B}$. If $|\Delta| \geq 2$ then we say that $\mathcal{B}$ is a \emph{non-singelton} block system.
\end{definition}
\begin{lemma} \label{coprime cycle} Let $\mathcal{B}$ be a non-singleton block system for $H$. Suppose that there exists $h \in H$ with a cycle $h_i$ of prime length, which is coprime to the lengths of all other cycles of $h$. Then there exists a block $\Delta$ of $\mathcal{B}$ such that $\supp{h_i} \subseteq \Delta$. In particular, $\Delta^h = \Delta$.
\end{lemma}
\begin{proof} Let $\Delta$ be a block containing at least one point $\alpha \in \supp{h_i}$, and let $\beta \in \Delta \backslash \{\alpha\}$. If $\beta \in \supp{h_i}$, then the result follows by Lemma \ref{comb}\ref{pcy}. If $\beta \notin \supp{h_i}$, then $\supp{h_i} \subseteq \Delta$ by Lemma \ref{comb}\ref{h1 cup h2}. 
\end{proof}
\section{Preliminary Results}
We begin by showing that Theorems \ref{newmain} and \ref{mainsize}(i) hold when $n \leq 11$ and prove Theorem \ref{mainsize}(ii). We then set up the notation for the rest of the paper, prove some preliminary lemmas and divide the task of proving Theorems \ref{newmain} and \ref{mainsize}(i) into subcases, see Hypothesis \ref{hypothesis}.
\begin{notation} \label{not} Throughout this and the next section let $G$ be either $\S{n}$ or $\A{n},$ acting on the set $\Omega = \{1,2,\ldots, n\}$ with $n \geq 4$. Let $\Omega_1 = \{1, 2, \ldots, k\}$ and $\Omega_2 = \{k+1, \ldots, n\}$ with $k >  n-k$. Let $M = \stab{G}{\Omega_1} = \stab{G}{\Omega_2}$. Then $M$ is isomorphic to $\Big(\S{k} \times \S{n-k} \Big) \cap G$. We let $x \in G \backslash M$.
\end{notation}
We first prove Theorem \ref{newmain} for some small values of $n$ and $n-k$. 
\begin{lemma} \label{coded} Let $n \leq 11$. Then Theorems \ref{newmain} and \ref{mainsize} hold.
\end{lemma}
\begin{proof}
Using $\textsc{Magma}$, we create a list of all possibilities for $x \in G \backslash M$, up to $M$-conjugacy. For each such $x$, we create a corresponding list $L$ of elements of $M$ up to conjugation by $ C_M(x)$. We then discard all $x$ for which there exists a $y \in L$ such that $\langle x, y \rangle =G$.

The only remaining $G,$ $M$ and $x$ are 
\begin{enumerate}[label=\rm{(\roman*)}]
\item $G = \S{n}$, $x=(1, k+1)$ and $(n,k) = (6,4),$ $(8,6),$ $(9,6),$ $(10,6)$ or $(10,8)$;
\item $(G, k, x) = (\S{4}, 3,(1, 4)(2, 3)),$ $(\A{5}, 3, (1, 4)(2, 3)),$ or $ (\A{6}, 4, (1, 5)(2, 6))$.
\end{enumerate}
In each case, $x$ is an involution and two involutions generate a dihedral group. Hence in these cases the maximal coclique in $\Gamma(G)$ containing $M$ is $M \cup x^M \backslash \{1\}$.
\end{proof}
\begin{proposition}\label{1 to k+1} Let $ n \geq 12$ and let $G$ and $M$ be as in Notation \ref{not}. Then $M$ is a maximal coclique of $\Gamma(G)$ if and only if for all $ x \in G \backslash M$ such that $1^x=k+1$ there exists $y \in M$ such that $\langle x, y \rangle =G$.
\end{proposition} 
\begin{proof}
The forward implication is clear, so assume that $M$ is not a maximal coclique of $\Gamma(G)$. Then there exists $x_1 \in G\backslash M$ such that $\langle x_1, y \rangle \neq G$ for all $y \in M$. Since $\A{m}$ is transitive for $m \geq 3$, there exists $h \in M$ such that $x_1^h$ maps 1 to $k+1$. Hence for all $y \in M$ we deduce that $\langle x_1^h, y^h \rangle  \neq G^h=G$.
\end{proof}
\begin{notation} For $y \in M$ define $$\ctm{y} := \mathcal{C}_1(y) \mid \mathcal{C}_2(y),$$ where $\mathcal{C}_i(y) := \ct{y\mid_{\Omega_i}}$ for $i=1,2.$
\end{notation}
We now define two distinct hypotheses which between them cover all possibilities in the case where $x \in G \backslash M$ is not a transposition and $n \geq 12$.
\begin{hypothesis} \label{hypothesis} Recall the set up of Notation \ref{not}. Let $n \geq 12$ so that $k \geq 7$.
\begin{enumerate}[label=\rm{(\Alph*)}]
\item Let $G = \A{n}$ if $n$ is odd and $G = \S{n}$ if $n$ is even.\label{anodd}
\item Let $G = \A{n}$ if $n$ is even and $G = \S{n}$ if $n$ is odd. \label{aneven}
\end{enumerate}
In both cases, assume that $1^x =k+1$ and that $x \neq (1, k+1)$.
\end{hypothesis}
We now prove two useful elementary lemmas that will help to simplify the proof of Theorem \ref{newmain}. Recall Definition \ref{Jor}, of a Jordan element. 
\begin{lemma} \label{n-k10} Let $n,G, M$, and $x$ be as in Case (A) or (B) of Hypothesis \ref{hypothesis}. If $|\supp{x} \cap \Omega_1| =1 $, the group $\langle x \rangle$ contains no Jordan element, and $n-k \leq 10$, then Theorem \ref{newmain} holds.
\end{lemma}
\begin{proof}
Since $|\supp{x} \cap \Omega_1|=1$, it follows that $|\supp{x}| \leq n-k+1 \leq 11$. Hence, since $x$ is not a Jordan element,
$$2(\sqrt{n}-1) < |\supp{x}|  \leq 11,$$ 
and so $12 \leq n \leq 42$. Notice that $|\supp{x} \cap \Omega_2| \geq 2$. Hence by 2-set transitivity of $\A{m}$ for $m \geq 3$, we may assume that $\{k+1,k+2\} \in \supp{x}$.  

If Hypothesis \hyperref[]{\ref*{hypothesis}.\ref*{anodd}} holds, then let $\mathcal{Y}$ be the set of $y \in M$ with $\ctm{y} =k \mid  (n-k-1)1$ and $\Theta_3=\{k+2\}$. If Hypothesis \hyperref[]{\ref*{hypothesis}.\ref*{aneven}} holds, then let $\mathcal{Y}$ be the set of elements of $ M$ with cycle type $k\mid (n-k)$. By Lemma \ref{oarity table}, $\mathcal{Y} \subseteq \A{n}$ if and only if $G = \A{n}$.

Since $\langle x \rangle$ contains no Jordan element, no power of $x$ is a cycle or the product of two transpositions. From this and the fact that $|\supp{x}| \leq 11$ there are few possible cycle structures for $x$. Using \textsc{Magma} it is easy to check that for each $x,n$ and $k$ there exists $y \in \mathcal{Y}$ such that $\langle x ,y \rangle=G$. 
\end{proof}
\begin{lemma} \label{small8} Let $n$, $G$, $M$ and $x$ be as in Case (A) or (B) of Hypothesis \ref{hypothesis}. Assume that $|\supp{x}| < 8$ or that $\mathcal{C}(x) \in T:= \{1^{(n-8)}\cdot 2 \cdot 3^2, \;\; 1^{(n-8)} \cdot 3 \cdot 5, \;\; 1^{(n-8)} \cdot 2^4, \;\; 1^{(n-9)} \cdot 3^3\}$. Then at least one of the following holds.
\begin{enumerate}[label=\rm{(\roman*)}]
\item The group $X = \langle x \rangle$ contains a Jordan element.
\item There exists an element $y \in M$ such that $\langle x, y \rangle =G$.
\end{enumerate}
\end{lemma} 
\begin{proof}
If $\ct{x} \notin S:=\{ 1^{(n-6)} \cdot 2^3,  \; 1^{(n-6)} \cdot 3^2, \; 1^{(n-8)} \cdot 2^4, \; 1^{(n-9)} \cdot 3^3\}$, then $X$ contains a Jordan element, so assume that $\ct{x} \in S$. If $n >30$ then $2(\sqrt{n}-1)>9$, and so $x$ is a Jordan element, so assume that $n \leq 30$. Since $|\supp{x}| >2$, and $\A{m}$ is 2-set transitive for $m \geq 3$, we may assume that either $\{k+1,k+2\} \in \supp{x}$ or $\{1,2\} \in \supp{x}$. 

Suppose that Hypothesis \hyperref[]{\ref*{hypothesis}.\ref*{anodd}} holds. If $k+2 \in \supp{x}$, then let $\mathcal{Y}$ be the set of elements $y \in M$ such that $\ctm{y} = k \mid (n-k-1)1$ with $\Theta_3 = \{k+2\}$. If $2 \in \supp{x}$, then let $\mathcal{Y}$ be the set of elements $y \in M$ such that $\ctm{y} = (k-1)1 \mid (n-k)$ with $\Theta_2 =  \{2\}$. If Hypothesis \hyperref[]{\ref*{hypothesis}.\ref*{aneven}} holds let $\mathcal{Y}$ be the set of elements $y \in M$ with $\ctm{y} = k \mid (n-k)$. By Lemma \ref{oarity table}, $\mathcal{Y} \subseteq \A{n}$ if and only if $G = \A{n}$.

In all cases, using $\textsc{Magma}$ it is easy to verify that there exists $y \in \mathcal{Y}$ such that $\langle x, y \rangle =G$.
\end{proof}
\section{Proof of Theorems \ref{newmain} and \ref{mainsize}}
In this section we complete the proofs of Theorems \ref{newmain} and \ref{mainsize}.
\subsection{Hypothesis \hyperref[]{\ref*{hypothesis}.\ref*{anodd}}}
In this section we show that under Hypothesis \hyperref[]{\ref*{hypothesis}.\ref*{anodd}} for all $x \in G \backslash M$ there exists $y \in M$ such that $\langle x, y \rangle =G$. We begin by putting restrictions on $x$.

\begin{lemma} \label{3} Let $n, G ,M$ and $x$ be as in Hypothesis \hyperref[]{\ref*{hypothesis}.\ref*{anodd}}. If $|\supp{x} \cap \Omega_1| =1$ and $x$ is a Jordan element, then there exists $y \in M$ such that $\langle x, y \rangle =G$.
\end{lemma}
\begin{proof}
By Hypothesis \ref{hypothesis}, there exists a point $t \in \supp{x} \backslash \{1, k+1\}$. Our assumption that $|\supp{x} \cap \Omega_1|=1$ implies that $t \in \Omega_2$.

By Lemma \ref{oarity table}, elements of $\S{n}$ composed of three cycles lie in $\A{n}$ if and only if $G =\A{n}$, so there exists $y \in M$ satisfy
$$\ctm{y} = k \mid (n-k-1)1,$$
with $\Theta_3 = \{t\}$. Let $H = \langle x, y\rangle$ and let $Y = \langle y \rangle$. Since $1 \in \Theta_1$ and $k+1 \in \Theta_2,$ it follows that $\Theta_1 \cup \Theta_2 = \Omega \backslash \{t\} \subseteq 1^H$. Since $t \in \supp{x}$, the group $H$ is transitive.

We show that $H$ is primitive. Let $\Delta$ be a non-singleton block for $H$ containing $t$. Let $a \in \Delta \backslash \{t\}$. Since $t$ is fixed by $y$, it follows that $\Delta^y = \Delta$. Hence $a^Y \cup \{t\} \subseteq \Delta$. 
If $a \in \Theta_1$, then $|\Delta| \geq k+1 > \frac{n}{2}$ and so $\Delta = \Omega$. If $a \in \Theta_2$, then $\Theta_2 \cup \{t\} \subseteq \Delta$. Since $\supp{x} \cap \Theta_1 = \{1\}$ and $( k+1)^{x^{-1}} =1 \neq t^{x^{-1}}$, it follows that $t^{x^{-1}} \in \Theta_2 \subseteq \Delta$. Hence $\Delta^{x^{-1}} =\Delta$, and so $\Delta^H = \Delta$. By the transitivity of $H$, it follows that $\Delta = \Omega$. 

Hence $H = \langle x, y \rangle$ is primitive, and contains the Jordan element $x$. Thus $\A{n} \leq H$, by Theorem \ref{jordanstyle}, and so $H = G$. \end{proof}

We now show that if $|\Omega_1 \cap \supp{x}|=1$, then there exists $y \in M$ such that $\langle x, y \rangle =G$.
\begin{lemma} \label{2}
Let $n, G ,M$ and $x$ be as in Hypothesis \hyperref[]{\ref*{hypothesis}.\ref*{anodd}}. If $|\supp{x} \cap \Omega_1| =1$, then there exists $y \in M$ such that $\langle x, y \rangle =G$.
\end{lemma}
\begin{proof} 
By Lemma \ref{3}, the result holds if $x$ is a Jordan element, and by Lemma \ref{n-k10} the result holds if $n-k \leq 10$. Hence we may assume that $n-k>10$ and that $|\supp{x}| \geq 2(\sqrt{n}-1)$. Thus $2(\sqrt{n}-1) \leq n-k+1$, so there exists a prime $\pnk$ as in Lemma \ref{pnk}. In addition, by Lemma \ref{small8} the result holds if $|\supp{x}| < 8$ or if $\mathcal{C}(x) = 1^{(n-8)} \cdot 2^4$, so we may assume otherwise. Hence we may let $s, t, u, v \in \supp{x} \backslash \{1, k+1\}$ be as in Lemma \hyperref[]{\ref*{8pts}.\ref*{8pts2}}. \\

The proof splits into two cases. First suppose that $\pnk \mid (n-k)$. By Lemma \ref{oarity table}, elements composed of five cycles lie in $\A{n}$ if and only if $G = \A{n}$, so there exists $y \in M$ such that 
$$\ctm{y} = k  \mid \pnk  (n-k-\pnk-2)  1^2,$$
with $s, t, t^x \in \Theta_2$, $k+1, s^x \in \Theta_3$, $\Theta_4 = \{u\}$ and $\Theta_5 = \{v\}$. Let $H = \langle x, y \rangle$. Since $1 \in \Theta_1$ and $k+1 \in \Theta_3$, it follows that $\Theta_1, \Theta_3 \subseteq1^H$. Then because $s \in \Theta_2$ and $s^x \in \Theta_3$, it follows that $\Theta_2 \subseteq1^H$. Since $(u, v)$ is not a cycle of $x$ and $\Omega \backslash \{u,v\}\subseteq1^H$, the group $H$ is transitive.

Let $\mathcal{B}$ be a non-singleton block system for $H$. Since $\pnk >2$ and $\pnk \mid (n-k)$, it follows that $\pnk \nmid (n-k-\pnk-2)$. Hence by Lemma \ref{coprime cycle}, there exists a block $\Delta \in \mathcal{B}$ such that $\Theta_2 \subseteq\Delta$. Therefore $\Delta^y = \Delta$. Furthermore, from $t, t^x \in \Theta_2$ we deduce that $\Delta^H = \Delta$, and hence $\Delta = \Omega$. Thus $H $ is primitive and contains the Jordan element $y^{k(n-k-\pb-2)}$, and so $H = G$. \\

Next suppose that $\pnk \nmid (n-k)$. By Lemma \ref{oarity table}, elements composed of three cycles lie in $\A{n}$ if and only if $G = \A{n}$, so there exists $y \in M$ such that
$$\ctm{y} = k \mid \pnk (n-k-\pnk),$$
with $s, t, t^x \in \Theta_2$ and $k+1, s^x \in \Theta_3$. Let $H = \langle x, y\rangle $. The argument that $H$ is transitive, primitive and contains a $\pnk$-cycle follows as in the previous case, and so $H = G$.
\end{proof}

We now complete the proof that under Hypothesis \hyperref[]{\ref*{hypothesis}.\ref*{anodd}} there exists $y \in M$ such that $\langle x, y \rangle =G$.

\begin{lemma} \label{1}
Let $n, G ,M$ and $x$ be as in Hypothesis \hyperref[]{\ref*{hypothesis}.\ref*{anodd}}. Then there exists $y \in M$ such that $\langle x, y \rangle =G$.
\end{lemma}
\begin{proof}
If $|\Omega_1 \cap \supp{x}| =1$, then the result holds by Lemma \ref{2}. Therefore we may assume that $|\supp{x} \cap \Omega_1| \geq 2$, and so there exists $t \in \big( \supp{x} \cap \Omega_1 \big) \backslash \{1\}$. Since $k \geq 7$, there exists a prime $\pb$ with $5 \leq \pb \leq k-2$, by Theorem \ref{BPC}.\\

First assume that $k=\pb+2$ and $n-k=\pb$. Hence $n=2\pb+2$. Thus $n$ is even and so $G = \S{n}$ by the assumption that Hypothesis \hyperref[]{\ref*{hypothesis}.\ref*{anodd}} holds. By Lemma \ref{oarity table}, elements of $\S{n}$ composed of three cycles are in $\S{n} \backslash \A{n}$. Let $y \in M$ satisfy 
$$\ctm{y} = 3(\pb-1)\mid \pb,$$
with $1 \in \Theta_1$, $t \in \Theta_2$ and $t^x \notin \Theta_2$. Let $H = \langle x, y\rangle$. Since $1^x = k+1$, it follows that $\Theta_1, \Theta_3 \subseteq 1^H$. Then $t \in \Theta_2$ and $t^x \in \Theta_1 \cup \Theta_3$, so $H$ is transitive.

Let $\mathcal{B}$ be a non-singleton block system for $H$. By Lemma \ref{coprime cycle}, there exists a block $\Delta \in \mathcal{B}$ with $\Theta_3 \subseteq\Delta$. Hence $\Delta^y = \Delta$ and so $\Delta$ is a union of the orbits of $y$ and contains $\Theta_3$. Since $|\Delta| \mid n$, it follows that $\Delta = \Omega$. Hence $H$ is primitive and contains the Jordan element $y^{3(\pb-1)}$, so $H = G$.\\

If $k \neq \pb +2$, then $k > \pb+2$, so for the remainder of the proof we may assume that 
\begin{equation}\label{assumption}
k - \pb > 2 \;\;\; \text{ or } \;\;\; n-k \neq \pb.
\end{equation}

Let $\mathcal{Y}$ be the set of elements $y \in M$ satisfying 
$$\ctm{y} = (k-\pb)\pb \mid (n-k)$$
with $1 \in \Theta_1$, $t \in \Theta_2$ and $t^x \notin \Theta_2$. 
By Lemma \ref{oarity table}, $\mathcal{Y} \neq \emptyset$, and consists of elements of $\A{n}$ if and only if $G = \S{n}$. For all $y \in \mathcal{Y},$ let $H = H(y) = \langle x, y\rangle $ and let $Y = \langle y \rangle$.
The proof of transitivity is identical to the previous case. We assume, by way of contradiction, that $H(y)$ is imprimitive for all $y \in \mathcal{Y}$, and let $\mathcal{B}$ be a non-trivial block system for $H$.

\smallskip 
First suppose, by way of contradiction, that there exists $\Delta_1 \in \mathcal{B}$ with $\Theta_2 \subseteq \Delta_1$. We begin by showing that if $\Theta_2 \subseteq \Delta_1$, then $\Delta_1 = \Theta_2$.
Suppose otherwise, and let $a \in \Delta_1 \backslash \Theta_2$. From $\Theta_2 \subseteq \Delta_1$ we see that $\Delta_1^y = \Delta_1$. If $a \in \Theta_1$, then $\Theta_1 \cup \Theta_2  \subseteq \Delta_1$ and so $|\Delta_1| \geq k > \frac{n}{2}$, a contradiction. Hence $a \in \Theta_3$, so $\Theta_2 \cup \Theta_3 \subseteq \Delta_1$, yielding the contradiction
$$|\Delta_1| \geq \pb + n-k > n-\frac{k}{2} > \frac{n}{2}.$$ 
Hence $\Delta_1 = \Theta_2$ and $\pb \mid n$. Since $\frac{n}{2} < k < 2 \pb$, it follows that $n< 4\pb,$ and consequently either $n =2\pb$ or $n=3\pb$. 

If $n=2\pb$, then $\mathcal{B}$ consists of two blocks $\Delta_1=\Theta_2 $ and $\Delta_2 = \Omega \backslash \Delta_1 = \Theta_1 \cup \Theta_3$. Since $1$ and $k+1 = 1^x \in \Delta_2$ both $x$ and $y$ leave $\Delta_2$ invariant, contradicting the transitivity of $H$.

If $n=3\pb$, then there exist blocks $\Delta_2$ and $\Delta_3$ such that $\mathcal{B} =\{ \Delta_1, \Delta_2, \Delta_3\}$. Hence $\Delta_2 \cup \Delta_3 = \Theta_1 \cup \Theta_3$. Since $\pb > \frac{k}{2}$, it follows that $|\Delta_2| = \pb $ does not divide $|\Theta_1|$. Hence $\Delta_2$ intersects both $\Theta_1$ and $\Theta_3$ non-trivially, and so $y^{\mathcal{B}} = (\Delta_2, \Delta_3)$. If there exists $\alpha \in \Delta_1$ such that $\alpha^x \in \Delta_1$, then $\Delta_1^x = \Delta_1= \Delta_1^y$, a contradiction. Therefore $\Delta_1^x \subseteq \Theta_1 \cup \Theta_3$ and $|\Delta_1| \geq 5$. Thus there exist distinct points $a_1, a_2 \in \Delta_1$ with $a_1^x, a_2^x$ both in $\Theta_1$ or both in $\Theta_3$. Let 
$$\mathcal{Y}_1 = \{y \in \mathcal{Y} \; | \; (a_1^x)^y = a_2\},$$
and notice that $\mathcal{Y}_1 \neq \emptyset$. Hence for all $y \in \mathcal{Y}_1$, the block $\Delta_2$ contains exactly one of $\{a_1^x ,a_2^x\}$. Thus $\emptyset \neq (\Delta_1^x \cap \Delta_2) \neq \Delta_2$, a contradiction.\\

Therefore if $n$ is even and $y \in \mathcal{Y}$, or if $n$ is odd and $y \in \mathcal{Y}_1$, there is no block $\Delta_1$ with $\Theta_2 \subseteq \Delta_1$. Hence it follows from Lemma \ref{comb}\ref{pcy} that $c_2^{\mathcal{B}}$ is a $\pb$-cycle. Let $\Delta \in \supp{c_2^{\mathcal{B}}}$. Since $\pb > k-\pb$ and $\Delta$ is non-trivial, it follows that $c_3^{\mathcal{B}}$ is also a $\pb$-cycle. Since $n-k < k <2 \pb,$ it follows that $\pb = n-k$ and so $|\Delta|=2$. Therefore $n$ is even and  $c_1^{\mathcal{B}}$ is a $\big( \frac{k-\pb}{2} \big)$-cycle.

From $\pb = n-k$ and \eqref{assumption}, it follows that $k - \pb > 2$. Therefore, since $c_1^{\mathcal{B}}$ is a $\big( \frac{k-\pb}{2} \big)$-cycle we deduce that there exists $a \in \Theta_1 \backslash \{1, t^{x^{-1}}\}$, and the set
$$\mathcal{Y}_a = \Bigg\{y \in \mathcal{Y} : 1^{y^{\frac{k-\pb}{2}}} =a\Bigg\} $$
is non-empty. For all $y \in \mathcal{Y}_a$, it follows that $\Delta_a = \{1, a\}$ is a block for $H(y)$. 
Consider $\Delta_a^x = \{k+1, a^x\}$. If $a^x \in \Omega_2$, then $\Delta_a^x \subseteq\Omega_2 = \Theta_3$, contradicting the fact that $c_3$ acts regularly on blocks.
Hence $a^x \in \Omega_1$. Since $a \neq t^{x^{-1}}$, it follows that $a^x \neq t$ and so there exists $y \in \mathcal{Y}_a$ such that $a^x \in \Theta_1$. Thus $k+1 \in \Delta_a^x \cap \Theta_3 $ and $a^x \in \Delta_a^x \cap \Theta_1$, contradicting the fact that $c_1$ and $c_3$ act on disjoint sets of blocks.

Hence there exists $y \in \mathcal{Y}_1$ or $y \in \mathcal{Y}_a$ such that $H = \langle x, y \rangle$ is primitive. If $n-k \neq \pb$, then $H$ contains the $\pb$-cycle $y^{(k-\pb)(n-k)}$. If $n-k = \pb$, then $H$ contains the $(k-\pb)$-cycle $y^{\pb}$. Thus in both cases $H = G$ by Theorem \ref{jordanstyle}.
\end{proof}

\subsection{Hypothesis \hyperref[]{\ref*{hypothesis}.\ref*{aneven}}}
In this section we show that for $n, G, M$ and $x$ as in Hypothesis \hyperref[]{\ref*{hypothesis}.\ref*{aneven}} there exists $y \in M$ such that $\langle x, y \rangle =G$. We begin with the case $|\Omega_1 \cap \supp{x}| = 2 = |\Omega_2 \cap \supp{x}|.$
\begin{lemma} \label{less than 4} Let $G, M,n$ and $x$ be as in Hypothesis \hyperref[]{\ref*{hypothesis}.\ref*{aneven}}. If $|\supp{x} \cap \Omega_1| = 2$ and $|\supp{x}  \cap \Omega_2| = 2$, then there exists $y \in M$ such that $\langle x, y \rangle =G$.
\end{lemma}
\begin{proof} 
Let $\supp{x} \cap \Omega_1 = \{1,t\}$ and $\supp{x} \cap \Omega_2 = \{k+1,r\}$. Then there are three possibilities for $x$, namely $(1, k+1, t, r), (1, k+1, r, t)$ or $(1, k+1)(t,r)$.

By Lemma \ref{oarity table}, elements of $\S{n}$ composed of two cycles lie in $\A{n}$ if and only if $G =\A{n}$, so there exists $y \in M$ such that
$$\ctm{y} =k \mid (n-k) ,$$
with $1^{y^2} = t $ and $(k+1)^y = r.$ Since $1^x = k+1$, it follows that $H= \langle x , y \rangle$ is transitive.

We prove that $H$ is primitive. Let $\Delta$ be a non-singleton block for $H$ containing 1. We shall show that there exists $b \in \Delta \cap \Theta_1$. Let $a \in \Delta \backslash \{1\}$. If $a \in \Theta_1$, then let $b:=a$. If $a \in \Theta_2$, then let $b:=1^{y^{(n-k)}}$. Since $k>n-k$, it follows that $b\neq 1$. From $a^{y^{(n-k)}} =a$ we deduce that $\Delta^{y^{(n-k)}} = \Delta$, hence $b \in \Delta \cap \Theta_1$. 

We claim that $\Delta^x = \Delta$ and so $k+1 \in \Delta$. If $b \in \fix{x}$, then this is immediate. If $b \notin \fix{x}$, then looking at $\supp{x}$ we deduce that $b =t = 1^{y^{2}}$. Hence $\Delta^{y^{2}} = \Delta$ and so $1^{y^4} \in \Delta$. Since $k \geq 7$, it follows that $1^{y^4} \neq 1, t$. Hence $1^{y^4} \in \fix{x}$ and so $\Delta^x = \Delta$.

The block $\Delta^y$ contains $r \in \supp{x}$ and $f :=1^y \in \fix{x}$. Therefore $(\Delta^y)^{x} = \Delta^y$ and $r^x \in \Delta^y$. Either $r^x=t = f^y$ or $r^x=1 = f^{y^{-1}}$. Hence either $\{f, f^y\}$ or $\{f, f^{y^{-1}}\} \subseteq \Delta^y$ hence $(\Delta^y)^y = \Delta^y$, and so $\Delta = \Omega$.

Therefore $H = \langle x, y \rangle$ is primitive. Furthermore, $H$ contains $x$, which is a Jordan element since $n \geq 12$. Therefore $\A{n} \leq H$ by Theorem \ref{jordanstyle} and so $H = G$.
\end{proof}
We now generalise to the case where both $|\Omega_1 \cap \supp{x}|$ and $|\Omega_2 \cap \supp{x}|$ are at least 2.
\begin{lemma}\label{4}
Let $n, G ,M$ and $x$ be as in Hypothesis \hyperref[]{\ref*{hypothesis}.\ref*{aneven}}. If $|\supp{x} \cap \Omega_1| \geq 2$ and $|\supp{x} \cap \Omega_2| \geq 2$, then there exists $y \in M$ such that $\langle x, y \rangle =G$.
\end{lemma}
\begin{proof}
By Lemma \ref{less than 4}, the result holds when $|\supp{x}|=4$, and so
we may assume that $|\supp{x}| > 4$. Hence there exist points $t \in \Omega_1 \backslash \{1\}$ and $r \in \Omega_2 \backslash \{k+1\}$ such that $t^x \neq r$. 

Let $\mathcal{Y}$ be the set of elements of $M$ composed of four cycles, $c_1$ and $c_2$ with support in $\Omega_1$, and $c_3$ and $c_4$ with support in $\Omega_2$, such that $1 \in \Theta_1, t \in \Theta_2$, $t^x \notin \Theta_2$, $k+1 \in \Theta_3$ and $\Theta_4 = \{r\}$. By Lemma \ref{oarity table}, elements of $\S{n}$ composed of four cycles lie in $\A{n}$ if and only if $G = \A{n}$, so $\mathcal{Y} \neq \emptyset$. For all $y \in \mathcal{Y}$, let $H =H(y)= \langle x, y\rangle$ and let $Y = \langle y \rangle$.

From $1^x = k+1$ we deduce that $\Theta_1, \Theta_3 \subseteq1^H$. Then $t \in \Theta_2$ and $t^x \in \Theta_1 \cup \Theta_3$ together imply that $\Omega \backslash \{r\} \subseteq1^H$. Since $r \in \supp{x}$, it follows that $H$ is transitive. Assume, by way of contradiction, that $H $ is imprimitive, and let $\mathcal{B}$ be a non-trivial block system for $H$.\\

Let $\pb$ be as in Theorem \ref{BPC}. We split into two cases. First assume that $\pb = n-k-1$ and $\pb= k-\pb+1$. Then $n = 3\pb$ and so it follows from Hypothesis \hyperref[]{\ref*{hypothesis}.\ref*{aneven}} that $G = \S{n}$. Let $$\mathcal{Y}_1 = \Big\{y \in \mathcal{Y} : \ctm{y} =  (\pb+1)(\pb-2) \mid \pb \cdot 1 \Big\}.$$
Then $\mathcal{Y}_1 \neq \emptyset$, and by Lemma \ref{coprime cycle}, there exists a block $\Delta \in \mathcal{B}$ with $\Theta_3 \subseteq \Delta$, so $|\Delta| \geq \pb$. Since $n = 3\pb$, it follows that $|\Delta| = \pb$ and $\Delta = \Theta_3$.
 Let $\Gamma$ be the block containing $r$, so $\Gamma^y = \Gamma$. Then $\Gamma$ is a union of some of the $\Theta_i$, a contradiction. Therefore for all $y \in \mathcal{Y}_1$, the group $H = \langle x, y \rangle$ is primitive. Furthermore, $H$ contains the Jordan element $y^{(\pb+1)(\pb-2)}$ and so $H=G$.\\

We may now assume that either 
\begin{equation}\label{?}
\pb \neq k-\pb+1 \;\;\;\; \text{or} \;\;\;\; \pb \neq n-k-1.
\end{equation} Let 
$$\mathcal{Y}_2 = \Big\{ y \in \mathcal{Y} : \ctm{y} = (k-\pb)  \pb \mid (n-k-1)  1 \Big\}.$$
Then $\mathcal{Y}_2 \neq \emptyset$.

We first show that there exists $\Delta \in \mathcal{B}$ with $ \Theta_2 \subseteq \Delta$. If $\pb \neq n-k-1$, then $\pb \nmid (n-k-1)$ by Lemma \ref{coprime}, and so this follows from Lemma \ref{coprime cycle}.
Suppose instead that $\pb = n-k-1$. If there exist blocks $\Delta_1, \ldots, \Delta_{\pb} \in \mathcal{B}$ such that $c_2^{\mathcal{B}} = (\Delta_1, \ldots, \Delta_{\pb})$, then $\Delta_i \cap \Theta_1 = \emptyset$ and $\Delta_i \cap \Theta_4 = \emptyset$ for $1 \leq i \leq \pb$ by Lemma \ref{comb}\ref{h1 cup h2}. Since $\mathcal{B}$ is non-trivial, it follows that $c_3^{\mathcal{B} }= (\Delta_1, \ldots, \Delta_{\pb})$ also, and so block size is two. Thus $|\Delta_1|=2$. Consider the block $\Gamma$ containing $r$. The point $r$ is fixed by $y$, so $\Gamma^y = \Gamma$, but $\Gamma \cap \Theta_1 \neq \emptyset$ so $|\Gamma| \geq k-\pb+1>2$, a contradiction. Hence $\Theta_2 \subseteq \Delta$ by Lemma \ref{comb}\ref{pcy}.

We show next that $c_1^{\mathcal{B}} = c_3^{\mathcal{B}}$.
From $|\Delta| \geq \pb > \frac{k}{2} > \frac{n}{4}$, it follows that $|\mathcal{B}|=2$ or $3$.
First suppose that $|\mathcal{B}|=2$, and let $\Gamma = \Omega \backslash \Delta$. Since $\Delta^y = \Delta$, it follows that $\Gamma^y = \Gamma$. If $\Theta_1 \subseteq \Delta$ or $\Theta_3 \subseteq \Delta$, then $|\Delta| > \frac{n}{2}$, and so $\Theta_1 \cup \Theta_3 \subseteq \Gamma$. Thus $1,k+1 \in \Gamma$ and $\Gamma^H = \Gamma$, a contradiction. We conclude that $|\mathcal{B}|=3$. If $\Delta$ contains a point of $\Theta_1$, then $\Theta_1 \cup \Theta_2 \subseteq\Delta$, a contradiction, so there exists a block $\Gamma \in \mathcal{B} \backslash \{\Delta\}$ containing a point of $\Theta_1$. 
Since $|\Theta_1| < |\Theta_2| \leq |\Delta|$, it follows that there exists a point $b \in \Gamma \backslash \Theta_1$. If $b\in \Theta_3,$ then $c_1^{\mathcal{B}} = c_3^{\mathcal{B}}$ by Lemma \ref{comb}\ref{sameinduced}. Hence assume for a contradiction that $b \notin \Theta_3$. It follows from $\Gamma \neq \Delta$ that $b \notin \Theta_2$. Hence $b = r$, so $\Gamma^y = \Gamma$. Therefore $\Gamma = \Theta_1 \cup \{r\}$, and the third block of $\mathcal{B}$ is $\Sigma = \Theta_3$. Since $|\Sigma| = |\Delta|,$ it follows that $\pb = n-k-1$. However, $|\Gamma| = k-\pb+1$, contradicting \eqref{?}.

If there exists $a \in \Delta$ such that $a^x \in \Delta,$ then $\Delta^H =\Delta,$ a contradiction. Therefore $\Theta_2^x \subseteq \Theta_1 \cup \Theta_3 \cup \{r\}$. By Theorem \ref{BPC}, $|\Theta_2| = \pb >5$. Hence there exist $s_1, s_2 \in \Theta_2$ such that either $s_1^x, s_2^x$ are both in $\Theta_1$ or both in $\Theta_3$. There exists $y \in \mathcal{Y}_2$ such that $s_1^{xy} = s_2^x$. Hence $(\Delta^{x})^y = \Delta^x$. Since $c_1^{\mathcal{B}} = c_3^{\mathcal{B}}$, it follows that $\Theta_1 \cup \Theta_3 \subseteq\Delta^x$. In particular, $\Delta^x$ contains 1 and $k+1$, and so $\Delta^{x^2} = \Delta^x = \Delta$. Hence $\Delta^H = \Delta$, a contradiction.

Hence for this $y$ the group $H = \langle x, y \rangle$ is primitive. If $\pb \neq n-k-1$, then $y^{(k-\pb)(n-k-1)}$ is a $\pb$-cycle and if $\pb = n-k-1$, then $y^{\pb}$ is a $(k-\pb)$-cycle. Hence in both cases $H = G$.
\end{proof}
We have reduced to the case of either $|\Omega_1 \cap \supp{x}|=1$ or $|\Omega_2 \cap \supp{x}|=1$. We first consider the case where $|\Omega_1 \cap \supp{x}|=1$.
\begin{lemma} \label{7}
Let $n, G ,M$ and $x$ be as in Hypothesis \hyperref[]{\ref*{hypothesis}.\ref*{aneven}}. If $|\supp{x} \cap \Omega_1| =1$, then there exists $y \in M$ such that $\langle x, y \rangle =G$.
\end{lemma}
\begin{proof}
First assume that $x$ is a Jordan element. It is immediate from Hypothesis \ref{hypothesis} that there exists $t \in \supp{x} \backslash \{1, k+1\}$, hence $t \in \Omega_2$. Let $s:=t^{x^{-1}}$. (Observe that we only define $k+1, (k+1)^y, (k+1)^{y^2}$ to be distinct when $|\supp{x} \cap \Omega_2| \geq 3$.)
By Lemma \ref{oarity table}, elements of $\S{n}$ composed of two cycles lie in $\A{n}$ if and only if $G =\A{n}$, so there exists $y \in M$ such that
$$\ctm{y} = k \mid (n-k),$$
with $(k+1)^y=t$, and if $s \neq k+1$, then $t^y= (k+1)^{y^2} =s$. Let $H = \langle x, y\rangle$. Since $1 \in \Theta_1$ and $k+1 \in \Theta_2$, it follows that $H$ is transitive.

Let $\mathcal{B}$ be a non-singleton block system for $H$, and let $\Delta \in \mathcal{B}$ with $1 \in \Delta$. It follows, just as in the proof of Lemma \ref{less than 4}, that there exists $b \in (\Delta \cap \Theta_1) \backslash \{1\}$. Since $\Theta_1 \cap \supp{x} = \{1\}$ and $|\Delta \cap \Theta_1| \geq 2$, it follows that $\Delta$ contains a point fixed by $x$, and so $\Delta^x = \Delta$. Therefore $k+1 = 1^x \in \Delta$ and $\{1^y, (k+1)^y \} = \{1^y, t\} \subseteq\Delta^y$. Since $1^y$ is fixed by $x$, it follows that $(\Delta^y)^{x^{-1}} = \Delta^y$, hence $s=t^{x^{-1}} \in \Delta^y$. From $t^y = s$ or $s^{y} =(k+1)^y = t$ we deduce that $\Delta^{y^2} = \Delta^y = \Delta$, and so $\Delta = \Delta^H = \Omega$.
Therefore $H$ is primitive. Furthermore, $H$ contains the Jordan element $x$, so $H = G$.\\

Hence we may assume that $x$ is not a Jordan element, and so $|\supp{x}| > 2(\sqrt{n}-1)$. By Lemma \ref{n-k10}, the result holds when $n-k \leq 10$, and so we may assume that $n-k >10$. Putting these two observations together, there exists a prime $\pnk$ as in Lemma \ref{pnk}. Furthermore, since the result holds when $x$ is a Jordan element, by Lemma \ref{small8} we may assume that $|\supp{x}| \geq 8$ and $\ct{x} \neq 1^{(n-8)} \cdot 2 \cdot 3^2,$ $1^{(n-8)} \cdot 3 \cdot 5$ or $1^{(n-9)} \cdot 3^3$. Hence let $ r,  s,  t$ be as in Lemma \hyperref[]{\ref*{8pts}.\ref*{8pts1}}.

If $\pnk \nmid (n-k-1)$, then let $i=1$, otherwise let $i=2$. Since $\pnk \leq n-k-4$, it follows that $n-k-\pnk-i \geq 2$. In addition, since $n-k \geq 11$, it follows that $n-k-i \geq 9$. Hence either $\pnk \geq 5$ or $n-k-\pnk-i \geq 5$. By Lemma \ref{oarity table}, elements of $\S{n}$ composed of four cycles lie in $\A{n}$ if and only if $G =\A{n}$, so there exists $y \in M$ such that 
$$\ctm{y}  = k \mid \pnk (n-k-\pnk-i)i,$$
with $r, t, t^x \in \Theta_2$, $k+1, r^x \in \Theta_3$, $s^x \in \Theta_4$, $s \in \Theta_2$ if $\pnk \geq 5$, and $s \in \Theta_3$ otherwise. Let $H = \langle x, y\rangle $. 
It is easy to see that $H$ is transitive. 

Let $\mathcal{B}$ be a non-singleton block system for $H$. By Lemma \ref{coprime cycle}, there exists $\Delta \in \mathcal{B}$ such that $\Theta_2 \subseteq \Delta$. Hence $\Delta^y = \Delta$. In addition, $\Delta$ contains $\{t, t^x\}$, so $\Delta^H = \Delta = \Omega$.
Hence $H$ is a primitive group containing the Jordan element $y^{k(n-k-\pnk-i)i}$, and so $H = G$.
\end{proof}
It remains to consider $|\supp{x} \cap \Omega_2|=1$. We first suppose that $x$ is a Jordan element.
\begin{lemma} \label{small in O1} Let $G, M, n$ and $x$ be as in Hypothesis \hyperref[]{\ref*{hypothesis}.\ref*{aneven}}. If $|\supp{x} \cap \Omega_2|=1$ and $x$ is a Jordan element, then there exists $y \in M$ such that $\langle x, y \rangle =G$.
\end{lemma}
\begin{proof}
It is immediate from Hypothesis \ref{hypothesis} that there exists $t \in \supp{x} \backslash \{1,k+1\}$. Our assumptions that $|\supp{x} \cap \Omega_2|=1$ and $1^x = k+1$ imply that $t, t^x \in \Omega_1$.

By Lemma \ref{oarity table}, elements of $\S{n}$ composed of two cycles lie in $\A{n}$ if and only if $G =\A{n}$, so there exists $y \in M$ such that
$$\ctm{y }  = k \mid (n-k),$$
with $1^y = t$, and $t^y = t^x$ if $t^x \neq 1$. It is clear that $H = \langle x, y \rangle$ is transitive. 

We assume, by way of contradiction, that $H$ is imprimitive, and let $\mathcal{B}$ be a non-singleton block system for $H$. Let $\Delta \in \mathcal{B}$ be the block containing $k+1$.
If $n-k=1$, then $\Delta^y = \Delta$, and so for $a \in \Delta \backslash \{k+1\}$ we find that $a^Y \cup \{k+1\} = \Omega = \Delta$, and so $H$ is primitive. Hence we assume now that $n-k \geq 2$.

We claim that $1 \in \Delta$. To see this, let $\Gamma \in \mathcal{B}$ be the block containing 1. If $\Gamma \cap \fix{x} \neq \emptyset$, then $k+1 = 1^x \in \Gamma$, hence $\Gamma=\Delta$. Similarly, if $\Delta \cap \fix{x} \neq \emptyset$, then $\Delta = \Gamma$. Hence we may assume that $\Delta, \Gamma \subseteq \supp{x}$. Since $|\Omega_2 \cap \supp{x}|=1$, it follows that $\Delta$ and $\Gamma$ both contain points of $\Theta_1$. Since $\Delta$ contains a point of $\Theta_2$, we deduce from Lemma \ref{comb}\ref{sameinduced} that $c_1^{\mathcal{B}} = c_2^{\mathcal{B}}$. However $|\Omega_2 \cap \supp{x}| =1$, so $\Delta = \Gamma$ and $1 \in \Delta$.

Notice that the block $\Delta^y$ contains $1^y = t$ and $(k+1)^y \in \fix{x}$. Hence $(\Delta^{y})^x = \Delta^y$ and in particular $\Delta^y$ contains both $t$ and $t^x$. If $t^x =1$, then $\Delta^y = \Delta$. If $t^x \neq 1$, then $\{t, t^x\} = \{t, t^y\} \subseteq \Delta^y = \Delta^{y^2} = \Delta$. Therefore in both cases $\Delta = \Delta^H = \Omega$.
Hence $H$ is primitive and contains the Jordan element $x$, and so $H = G$.
\end{proof}
Finally, we generalise to the case $|\supp{x} \cap \Omega_2| =1$.
\begin{lemma} \label{8ish} Let $n, G ,M$ and $x$ be as in Hypothesis \hyperref[]{\ref*{hypothesis}.\ref*{aneven}}. If $|\supp{x} \cap \Omega_2| = 1$, then there exists $y \in M$ such that $\langle x, y \rangle =G$.
\end{lemma}
\begin{proof} 
First assume that $k \geq10$, so there exists a prime $\pk$ as in Lemma \ref{pk}.
If $x$ is a Jordan element, then the result holds by Lemma \ref{small in O1}. Hence by Lemma \ref{small8} the result holds if $|\supp{x}| < 8$ or $\ct{x} = 1^{(n-8)} \cdot 2 \cdot 3^2, 1^{(n-8)} \cdot 3 \cdot 5$ or $1^{(n-9)} \cdot 3^3$, so assume otherwise. Thus there exist $ r,   s, t\in \supp{x} $ as in Lemma \hyperref[]{\ref*{8pts}.\ref*{8pts1}}. 

Let $i=1$ if $\pk \nmid (k-1)$ and $i=2$ otherwise. Then $k -i- \pk \geq 3$. By Lemma \ref{oarity table}, elements of $\S{n}$ composed of four cycles lie in $\A{n}$ if and only if $G =\A{n}$, so there exists $y \in M$ such that
$$\ctm{y } =(k-i-\pk)  \pk  i \mid (n-k),$$
with $1 , r,s \in \Theta_1,$ $r^x, t, t^x \in \Theta_2$ and $ s^x \in \Theta_3$. Let $H = \langle x, y\rangle $. Then it is easy to check that $H$ is transitive.

Let $\mathcal{B}$ be a non-singleton block system for $H$. By Lemma \ref{coprime cycle}, there exists $\Delta \in \mathcal{B}$ such that $\Theta_2 \subseteq\Delta$, hence $\Delta^y = \Delta$. In addition, $t, t^x \in \Delta$, and so $\Delta^x = \Delta = \Omega$, and hence $H$ is primitive. Furthermore, $H$ contains the $\pk$-cycle $y^{(k-i-\pk)i(n-k)}$ and so $H = G$.\\

Now suppose that $k \leq 9$. It is immediate from Hypothesis \ref{hypothesis} that $7 \leq k \leq 9$ and so $12 \leq n \leq 17$. From $|\supp{x} \cap \Omega_2|=1$, it follows that $|\supp{x}| \leq k+1 \leq 10$. 
We verify in $\textsc{Magma}$ that for each $x$ there exists $y \in M$ such that $\langle x, y \rangle =G$.
\end{proof}
\begin{lemma} \label{caseb} Let $n, G, M$ and $x$ be as in Hypothesis \hyperref[]{\ref*{hypothesis}.\ref*{aneven}}. Then there exists $y \in M$ such that $\langle x, y \rangle =G$.
\end{lemma}
\begin{proof}
If $|\supp{x} \cap \Omega_1|=1$ or $|\supp{x} \cap 
\Omega_2|=1$, then the result holds by Lemma \ref{7} and \ref{8ish}, respectively. Otherwise, $|\supp{x} \cap \Omega_i| \geq 2$ for $i \in \{1,2\}$, so the result holds by Lemma \ref{4}.
\end{proof}
\subsection{Completing the proof of Theorems \ref{newmain} and \ref{mainsize}}
In Lemmas \ref{1} and \ref{caseb} we prove that if $n \geq 12$ and $x \in G \backslash M$ is not a transposition, then there exists $y \in M$ such that $\langle x, y \rangle =G$. Here we show that if $x \in G \backslash M$ is a transposition, then there exists $y \in M$ such that $\langle x, y \rangle =G$ if and only if $\gcd(n,k)=1$, completing the proof of Theorem \ref{newmain}. We also complete the proof of Theorem \ref{mainsize}.
\begin{theorem} \label{1, k+1, coprime} Let $n,k,G = \S{n}$ and $M$ be as in Notation \ref{not}, and let $x  \in G \backslash M$ be a transposition. Then there exists $y \in M$ such that $\langle x, y \rangle =G$ if and only if $\gcd(n, k)=1$.
\end{theorem}
\begin{proof} By Proposition \ref{1 to k+1}, it suffices to consider $x = (1, k+1)$.

First assume that $\gcd(n,k)=1$. Let $y \in M$ with $\ctm{y}=k \mid (n-k)$, and let $H = \langle x, y \rangle$. It is clear that $H$ is transitive. Let $\mathcal{B}$ be a non-singleton block system for $H$, let $\Delta \in \mathcal{B}$ with $1 \in \Delta$, and let $a \in \Delta \backslash \{1\}$.
If $a \in \Omega_1$, then $a^x = a$ and so $\Delta^x = \Delta$. Hence $k+1 = 1^x \in \Delta$. Therefore, without loss of generality, $a \in \Omega_2$. Thus $a^{y^{(n-k)}} = a$, and so $\Delta^{y^{(n-k)}} = \Delta$. Therefore $1^{\langle y^{(n-k)} \rangle } \subseteq \Delta$. It follows from $\gcd(n, k)=1$ that $1^{\langle y^{(n-k)} \rangle } = \Omega_1$. Hence $|\Delta| \geq k +1> \frac{n}{2}$, so $\Delta=\Omega$. 
Hence $H$ is primitive, and contains the Jordan element $x$. Since $x \in \S{n} \backslash \A{n}$, it follows that $H = \S{n}$.\\

Next assume that $\gcd(n, k)=t>1$. Let $y \in M$ be such that $\langle x, y \rangle$ is transitive. Then $\ctm{y}=k \mid (n-k)$. We claim that the set of translates of 
$\Delta = 1^{\langle y^{\frac{k}{t}} \rangle} \cup (k+1)^{\langle y^{\frac{n-k}{t}} \rangle}$ form a proper non-trivial block system for $\langle x, y \rangle$, so that $\langle x, y \rangle \neq \S{n}$. To see this, notice that $|\Delta| = t >1$. Also, note that $\dot{\bigcup}_{i=0}^{t-1} \Delta^{y^i} = \Omega$ and $x$ fixes setwise $\Delta^{y^i}$ for $0 \leq i \leq t-1$.
\end{proof}
\vspace{0.3cm}
\noindent
\textit{Proof of Theorem \ref{newmain}.}
The subgroup $M$ is a maximal coclique in $\Gamma(G)$ if and only if for all $x \in G \backslash M$ there exists $y \in M$ such that $\langle x, y \rangle =G$, so let $x \in G \backslash M$. Then by Proposition \ref{1 to k+1} we may assume without loss of generality that $1^{x} = k+1$. 

If $n \leq 11$, then the result holds by Lemma \ref{coded}, so assume that $n \geq 12$. If Hypothesis \hyperref[]{\ref*{hypothesis}.\ref*{anodd}} holds, then the result follows from Lemma \ref{1}, and if Hypothesis \hyperref[]{\ref*{hypothesis}.\ref*{aneven}} holds, then the result follows from Lemma \ref{caseb}. If neither part of Hypothesis \ref{hypothesis} holds, then $x = (1,k+1)$, so the result follows from Theorem \ref{1, k+1, coprime}. \hfill \qed 
\vspace{0.3cm}\\
\noindent
\textit{Proof of Theorem \ref{mainsize}.}
Parts (i)(b), (ii)(a) and (ii)(b) follow immediately from Lemma \ref{coded}.
It remains to prove (i)(a), so let $G = \S{n}$ and $\gcd(n,k)>1$.

Let $C$ be a maximal coclique in $\Gamma(G)$ containing $M$. Theorem \ref{newmain} proves that $C \neq M \backslash \{1\}$. Lemmas \ref{1} and \ref{caseb} show that if $x \in G \backslash M$ is not a transposition, then $x \notin C$. Hence $M \backslash \{1\} \subsetneq C \subseteq M \cup (1, k+1)^M \backslash \{1\}$. By Theorem \ref{1, k+1, coprime}, for all $y,m \in M$, the group $\langle y, (1,k+1)^m \rangle $ is not equal to $G$. For $n > 3$ no two transpositions generate $G$ so $M \cup (1, k+1)^M  \backslash \{1\} \subseteq C$. Therefore $C = M \cup (1, k+1)^M \backslash \{1\}$, as required. \hfill \qed
\section{Proof of Theorem \ref{AGLstuff}}
The methods here are different to those in Section 3, because the maximal subgroups of $\S{p}$ and $\A{p}$ are classified.
We first consider an exceptional case.
\begin{lemma} \label{M23} The group $M_{23}$ is a maximal coclique in $\A{23}$.
\end{lemma}
\begin{proof}
Let $G = \A{23}$. A quick calculation in \textsc{Magma} shows that the only transitive maximal subgroups of $G$ are two conjugacy classes of groups isomorphic to $M_{23}$, which we denote $\mathcal{A}$ and $\mathcal{B}$. Since $\mathcal{A}$ and $\mathcal{B}$ are conjugate in $\S{23}$ it suffices to consider $M \in \mathcal{A}$. Recall that the Sylow 23-subgroups of $\A{23}$ are cyclic and transitive. 

First suppose that the order of $x$ is at least 4. We claim that there exists $Z \in \Syl_{23}(M)$ such that $\langle x, Z \rangle =G$. By calculating the permutation character of $A_{23}$ on the cosets of $M_{23}$ in $\textsc{Magma}$, we see that $x$ lies in at most 4608 groups $B \in \mathcal{B}$, and each element of order 23 lies in exactly one $A \in \mathcal{A}$ and exactly one $B \in \mathcal{B}$. 
Let $Z \in \Syl_{23}(M)$, since $M \in \mathcal{A}$ it follows from \cite{Atlas}, that $N_M(Z) = N_G(Z)$ and $N_M(Z) \leq_{\textrm{max}} M$. Hence $|\Syl_{23}(M)| = [M : N_M(Z)] =40320$, and so there are $40320 -4608 = 35712$ possibilities for $Z \in \Syl_{23}(M)$ such that $H:=\langle x, Z \rangle$ is contained in no $B \in \mathcal{B}$. Since $x \notin M$, and $M$ is the unique subgroup of $\mathcal{A}$ containing $Z$, it follows that $H = G$.

Now suppose that $x$ has order 2 or 3 and let $Z \in \Syl_{23}(M)$. By the previous case, $M$ is the unique group of $\mathcal{A}$ containing $Z$ and there exists a unique $B \in \mathcal{B}$ with $Z \leq B$. Therefore if $x \notin B$ then $\langle x, Z \rangle =G$. Hence suppose that $x \in B$ and proceed using \textsc{Magma}. Let $M$ be the representative of one conjugacy class of $M_{23}$ in $G$, and let $B_0$ be the representative of the other. Then $B$ can by found by conjugating $B_0$ by the element of $S_{23}$ which conjugates a subgroup of $\Syl_{23}(B_0)$ to $Z$. It is then possible to check that for each element $x \in B \backslash M$ of order 2 or 3, there exists $y \in M$ such that $\langle x, y \rangle =G$.
\end{proof}
The following theorem enables us to classify the maximal subgroups of $\S{p}$ and $\A{p}$.
\begin{theorem}[{\cite[p.99]{DM}}\label{DM}] 
A transitive group of prime degree $p$ is one of the following:
\begin{enumerate}[label=\rm{(\roman*)}]
\item the symmetric group $\S{p}$ or the alternating group $\A{p}$;
\item a subgroup of $\AGL_1(p)$;
\item a permutation representation of $\PSL_2(11)$ of degree 11;
\item one of the Mathieu groups $M_{11}$ or $M_{23}$ of degree 11 or 23, respectively;
\item a group $G$ with $\PSL_d(q) \leq G \leq \PGamL_d(q)$ of degree $p = \frac{q^d-1}{q-1}$.
\end{enumerate}
\end{theorem}
In the following lemma we collect some standard facts about $\AGL_1(p)$. 
\begin{lemma} \label{AGL}
Let $G = \S{p}$ and $M = \AGL_1(p) \leq G$. 
\begin{enumerate}[label=\rm{(\roman*)}]
\item The group $M$ is sharply 2-transitive. \label{2trans}
\item $M$ has a unique Sylow $p$-subgroup, $P = \langle z \rangle$, and $M=N_G(P) \cong C_p : C_{p-1}$. \label{norm}
\item The elements of $M$ are $p$-cycles or powers of $(p-1)$-cycles. \label{fixedpoints}
\item If $y_1, y_2 \in M$ are $(p-1)$-cycles such that $\langle y_1 \rangle \neq \langle y_2 \rangle$, then $M = \langle y_1, y_2 \rangle$. \label{2gen}

\end{enumerate}
\end{lemma} 
We now have the tools required to prove Theorem \ref{AGLstuff}.\\

\noindent
\textit{Proof of Theorem \ref{AGLstuff}.}
Since $p$ is prime, for all $k$ with $p > k > \frac{p}{2}$, it follows that $\gcd(k, p-k) =1$. If $G = \S{p}$, then by Theorem \ref{newmain} each intransitive maximal subgroup is a maximal coclique. If $G = \A{p}$, then for $p \neq 5$ each intransitive maximal subgroup is a maximal coclique, and if $p=5$, then $(\S{4} \times \S{1} ) \cap \A{5}$ is a maximal coclique but $(\S{2} \times \S{3}) \cap \A{5}$ is not.
 
If $p=11$ or $23$ and $G = \A{p}$, then the transitive maximal subgroups are the respective Mathieu groups. If $p=11$, then the result follows from a straightforward $\textsc{Magma}$ calculation, similar to the one described in the proof of Lemma \ref{coded}. The result for $p=23$ follows from Lemma \ref{M23}. Hence assume from now on that if $G = \A{p}$, then $p \neq 11, 23$. 

Let $G = \S{p}$, let $M= \A{p}$ and let $x \in G \backslash M$. Let $y \in M$ be a $p$-cycle such that $y$ is not normalized by $x$. Then $ \langle x, y \rangle$ is a transitive subgroup and lies in no conjugate of $\AGL_1(p) \cap G$ by Lemma \hyperref[]{\ref*{AGL}.\ref*{norm}}. Hence $\A{p} \leq \langle x, y \rangle =G$, and so $M$ is a maximal coclique.\\

By Theorem \ref{DM} the only remaining case is $M = \AGL_1(p) \cap G$.
First consider together the cases $G = \A{p}$, or $G = \S{p}$ and $x \notin M$ is an odd permutation. Let $y \in M$ be a $p$-cycle, so $H = \langle x, y \rangle$ is transitive. By Lemma \hyperref[]{\ref*{AGL}.\ref*{norm}}, $y$ is contained in no other conjugate of $M = N_G(\langle y \rangle)$. Since $x \notin M$, it follows that $H \neq M$, and so $H = G$.

Assume instead that $G = \S{p}$ and $x \notin M$ is an even permutation. First let $x$ be of order $p$. Let $y_1, y_2 \in M$ be $(p-1)$-cycles with $\langle y_1 \rangle \neq \langle y_2 \rangle$. Then $H_1 = \langle x, y_1 \rangle$ and $H_2 = \langle x, y_2 \rangle$ are distinct transitive subgroups of $G$. Note that $y_1, y_2 \in G \backslash \A{p}$, and so $H_1$ and $H_2$ either conjugate to $M$, or equal to $G$. In the latter case the result holds, so assume that both $H_1$ and $H_2$ are conjugate to $M$. Since $x \in H_1 \cap H_2$ and $N_G(\langle x \rangle)$ is the unique conjugate of $M$ containing $x$, it immediately follows that $H_1 = N_G(\langle x \rangle) = H_2$, a contradiction.

Assume next that $x$ lies in no conjugate of $M$. Let $t \in \supp{x}$ and let $y$ be a $(p-1)$-cycle of $M$ fixing $t$. Then $ \langle x, y \rangle$ is transitive and contained in no conjugate of $M$, and so $\langle x, y \rangle = G$.

Finally assume that $x$ is an even permutation, not a $p$-cycle and lies in some conjugate of $M$. By Lemma \hyperref[]{\ref*{AGL}.\ref*{fixedpoints}}, $x$ is a proper power of a $(p-1)$-cycle. We claim there exists a $(p-1)$-cycle $y$ in $M$, and $z \in \langle y \rangle,$ such that $H = \langle x, y \rangle$ is transitive and $1< \fix{z^{-1}x} <p$. Since, by Lemma \hyperref[]{\ref*{AGL}.\ref*{2trans}}, each non-identity element of $M$ has at most one fixed point it will follow that $H$ lies in no conjugate of $M$, and so $H = G$. 

It remains to prove the claim. Since $x$ is a proper power of a $(p-1)$-cycle, $x$ has one fixed point which we shall call $f$. Let $M_f$ denote the point stabilizer of $f$ in $M$, and $P$ denote the cyclic $p$-subgroup of $M$.

Since $p \geq 5,$ there exist $a, b \in \supp{x}$ with $a \neq b$. By sharp 2-transitivity there exists an element $y_1$ in $M$ such that $a^{y_1} = a^x$ and $b^{y_1} = b^x$. If $y_1 \notin M_f \cup P$, then $y_1$ lies in a cyclic subgroup $\langle y \rangle$ of order $(p-1)$ and $H = \langle x, y \rangle$ is transitive. In addition $a,b \in \fix{y_1^{-1}x}$, as claimed.

Suppose instead that $y_1 \in  M_f \cup P$. Since $y_1 \neq x$ and $p \geq 5$, there exists $c \in \supp{x}$ with $c \neq a,b$ such that $c^{y_1} \neq c^x$. By sharp 2-transitivity, there exists $y_2 \in M$ such that $a^{y_2} = a^x$ and $c^{y_2} =c^x$. If $y_2 \notin M_f \cup P$, then the result follows as for $y_1$ with $a,c \in \fix{y_2^{-1}x}$.

Suppose that $y_1, y_2 \in  M_f \cup P$. It follows from $c^{y_1} \neq c^{y_2}$ that $y_1 \neq y_2$. Therefore because $a^{y_1} = a^{y_2}$, by sharp 2-transitivity, it follows that $b^{y_2} \neq b^{y_1} = b^x$. There is a unique element of $ M_f $, and a unique element of $P$, sending $a$ to $a^x$. Let $Y_1$ and $Y_2$ be the maximal cyclic subgroups containing $y_1$ and $y_2$. Then $Y_1 \cup Y_2 = M_f \cup P$.

Since $M$ is sharply 2-transitive, there exists $y_3 \in M$ such that $b^{y_3} =b^x$ and $c^{y_3} =c^x$. Since $y_1$ is the unique element of $Y_1$ sending $b$ to $b^x$, and $c^{y_1} \neq c^{y_3}$, it follows that $y_3 \notin Y_1$. Since $y_2$ is the unique element of $Y_2$ sending $c$ to $c^x$ and $b^{y_2} \neq b^{y_3}$, it follows that $y_3 \notin Y_2$. Hence $y_3 \notin Y_1 \cup Y_2 = M_f \cup P$. Thus let $y \in M$ be a $(p-1)$-cycle such that $y^t = y_3$ for some $t \in \mathbb{N}$.  Then $y$ satisfies the claim with $b,c \in \fix{y^{-t}x} = \fix{y_3^{-1}x}$. Therefore the claim and the theorem follow. \hfill \qed

Veronica Kelsey \& Colva M. Roney-Dougal: Mathematical Institute, Univ. St Andrews, KY16 9SS, UK\\
vk49@st-andrews.ac.uk\\
Colva.Roney-Dougal@st-andrews.ac.uk
\end{document}